\documentclass{article}
\usepackage{amssymb}
\usepackage{amsmath}

\newtheorem{theorem}{Theorem}

\newtheorem{lemma}[theorem]{Lemma}

\newenvironment{proof}[1][Proof]{\noindent\textbf{#1.} }{\ \rule{0.5em}{0.5em}}
\input{tcilatex}
\begin{document}

\title{A general formula for determinants and inverses of $r$-circulant
matrices with third order recurrences}
\author{{\small Emrullah KIRKLAR\thanks{%
Corresponding author} \thanks{%
E-mails: e.kirklar@gazi.edu.tr, fatihyilmaz@gazi.edu.tr},} {\small Fatih
YILMAZ\ } \\
{\small Polatl\i\ Art and Science Faculty, Gazi University, Turkey}}
\maketitle

\begin{abstract}
This note provides formula for determinant and inverse of $r$-circulant
matrices with general sequences of third order. In other words, the study
combines many papers in the literature.

\textbf{AMS2010: }15A09; 15A15.
\end{abstract}

\section{Introduction}

A $r$-circulant matrix of order $n$, $C_{n}:=circ_{r}\,(c_{0},c_{1},\ldots
,c_{n-1})$, associated with the numbers $c_{0},c_{1},\ldots ,c_{n-1}$, is
defined as 
\begin{equation}
C_{n}=\left( 
\begin{array}{ccccc}
c_{0} & c_{1} & \ldots & c_{n-2} & c_{n-1} \\ 
rc_{n-1} & c_{0} & \ldots & c_{n-3} & c_{n-2} \\ 
\vdots & \vdots & \ddots & \vdots & \vdots \\ 
rc_{2} & rc_{3} & \ldots & c_{0} & c_{1} \\ 
rc_{1} & rc_{2} & \ldots & rc_{n-1} & c_{0}%
\end{array}%
\right) .  \label{2}
\end{equation}%
where each row is a cyclic shift of the row above it \cite{1}. If $r=1$,
then the matrix $C_{n}$ is ordinary circulant matrix. If $r=-1$, then the
matrix $C_{n}$ is skew-circulant matrix.

Circulant matrices and their applications are a fundamental key in many
areas of pure and applied science (see \cite{10,15}, and references there
in). Recently, many researcher get very interesting properties of them. For
example, in \cite{1}, Shen and Cen obtained upper and lower bounds for the
spectral norms of $r$-circulant matrices involving Fibonacci and Lucas
numbers. Further, they gave some bounds for the spectral norms of Kronecker
and Hadamard products of these matrices. In \cite{2}, Shen et al. obtained
useful formulas for determinants and inverses of circulant matrices with
Fibonacci and Lucas numbers, using properties of circulant matrices and this
sequences. In \cite{3}, Bozkurt and Tam gave formulas for determinants and
inverses of circulant matrices involving Jacobsthal and Jacobsthal-Lucas
numbers taking into account the method in \cite{2}. Bozkurt and Tam \cite{4}
defined $r$-circulant matrices with general second order number sequences.
Then, the authors obtained formulas for determinant and inverse of this
matrix. Moreover, they gave some bounds for norms of $r$-circulant matrices
involving Fibonacci and Lucas numbers. Yazl\i k and Taskara \cite{5}
considered circulant matrices with $k$-Horadam numbers. Then, the authors
obtained formulas for determinant and inverse of this matrix. Liu and Jiang 
\cite{13} defined Tribonacci circulant matrix, Tribonacci left circulant
matrix, Tribonacci $g$-circulant matrix. Then, the authors acquired
determinants and inverses of these matrices. In \cite{14}, Bozkurt et al.
considered the determinant of circulant and skew-circulant matrices whose
entries are Tribonacci numbers. Bozkurt and Y\i lmaz \cite{16} obtained
formulas for determinant and inverse of circulant matrices with Pell and
Pell-Lucas numbers.

In this paper, we consider third order linear recurrence for $n>2$: 
\begin{equation}
W_{n}=pW_{n-1}+qW_{n-2}+tW_{n-3}\,  \label{1}
\end{equation}%
$~$with initial conditions $W_{0}=0,W_{1}=a~$and $W_{2}=b$. The first few
values are 
\begin{equation*}
0,a,b,pb+qa,p^{2}b+pqa+qb+ta,\ldots \,.
\end{equation*}%
Then, we obtain formulas for determinants and inverses of $r$-circulant
matrices $E_{n},$ i.e.,%
\begin{equation*}
E_{n}:=circ_{r}\,(W_{1},W_{2},\ldots ,W_{n}),
\end{equation*}%
where $W_{n}$ is given by \eqref{1}.

As it can be seen from the definition of the sequence, it is a general form
of some well-known sequences. In other words,

$\diamondsuit $ If $p=q=a=b=r=1$ and $t=0$, then we obtain determinant and
inverse of circulant matrices with Fibonacci numbers, as in \cite{2}.

$\diamondsuit $ If $p=a=b=r=1,t=0$ and $q=2$, then we obtain determinant and
inverse of circulant matrices with Jacobsthal numbers, as in \cite{3}.

$\diamondsuit $\ If $q=a=r=1$, $p=b=2$ and $t=0$, then we obtain determinant
and inverse of circulant matrices with Pell numbers, as in \cite{16}.

$\diamondsuit $ If $p=q=a=b=t=r=1$, then we obtain determinant and inverse
of circulant matrices with Tribonacci numbers, as in \cite{13}.

$\diamondsuit $ If $p=q=a=b=t=1$ and $r=-1$, then we obtain determinant and
inverse of skew-circulant matrices with Tribonacci numbers, as in \cite{14}.

To sum up, the derived formulas combine many of the papers in the literature.

\section{Determinant of $E_{n}$}

This section is dedicated for determinant formula of $r$-circulant matrices
with general third order sequences. Firstly, let us give the following
lemmas.

\begin{lemma}
\cite{14} If 
\begin{equation}
D_{n}=\left( 
\begin{array}{cccccc}
d_{1} & d_{2} & d_{3} & \cdots & d_{n-1} & d_{n} \\ 
a & b &  &  &  &  \\ 
c & a & b &  &  &  \\ 
& c & a & \ddots &  &  \\ 
&  & \ddots & \ddots & \ddots &  \\ 
&  &  & c & a & b%
\end{array}%
\right) ,  \label{60}
\end{equation}%
then 
\begin{equation}
\det D_{n}=\sum_{k=1}^{n}d_{k}b^{n-k}\left( -\sqrt{bc}\right)
^{k-1}U_{k-1}\left( \frac{a}{2\sqrt{bc}}\right) \,,  \label{detAn}
\end{equation}%
where $U_{k}(x)$ is the $k$th Chebyshev polynomial of second kind.
\end{lemma}

\begin{lemma}
If%
\begin{equation*}
B_{n}=\left( 
\begin{array}{ccccccc}
X_{1} & d_{1} & d_{2} & d_{3} & \cdots & d_{n-1} & d_{n} \\ 
Y_{1} & f_{1} & f_{2} & f_{3} & \cdots & f_{n-1} & f_{n} \\ 
0 & a & b & 0 &  &  & 0 \\ 
& c & a & b &  &  &  \\ 
&  & c & a & \ddots &  &  \\ 
&  &  & \ddots & \ddots & \ddots & 0 \\ 
0 &  &  & 0 & c & a & b%
\end{array}%
\right) ,
\end{equation*}

then 
\begin{equation*}
\det (B_{n})=X_{1}\sum_{k=1}^{n-1}f_{k}b^{n-1-k}\left( -\sqrt{bc}\right)
^{k-1}U_{k-1}\left( \frac{a}{2\sqrt{bc}}\right)
-Y_{1}\sum_{k=1}^{n-1}d_{k}b^{n-1-k}\left( -\sqrt{bc}\right)
^{k-1}U_{k-1}\left( \frac{a}{2\sqrt{bc}}\right) ,
\end{equation*}

where $U_{k}(x)$ is the $k$th Chebyshev polynomial of second kind.
\end{lemma}

\begin{proof}
Using the same method in the first Lemma 1, we have%
\begin{equation*}
\det (B_{n})=X_{1}\det (F_{n-1})-Y_{1}\det (D_{n-1})~~.
\end{equation*}

So, this proof is completed.
\end{proof}

\begin{theorem}
For $n\geq 4,$ the determinant of $E_{n}$ is 
\begin{eqnarray*}
&&{\small W}_{1}\left[ {\small (g}_{n}{\small +jf}_{n}{\tiny )}\left( (%
{\small W}_{1}{\small -r(pW}_{n}{\small +qW}_{n-1}{\small ))x}_{n}^{n-3}+%
{\small rt}\sum_{k=2}^{n-2}{\small W}_{n-1-k}{\small x}_{n}^{n-2-k}\left( -%
\sqrt{{\small x}_{n}{\small z}_{n}}\right) ^{k-1}{\small U}_{k-1}\left( 
\frac{{\small y}_{n}}{{\small 2}\sqrt{{\small x}_{n}{\small z}_{n}}}\right)
\right) \right. \\
&&\left. {\small -h}_{n}\sum_{k=1}^{n-2}{\small [rW}_{n+1-k}{\small -(pr-j)W}%
_{n-k}{\small ]x}_{n}^{n-2-k}\left( -\sqrt{x_{n}z_{n}}\right) ^{k-1}{\small U%
}_{k-1}\left( \frac{{\small y}_{n}}{{\small 2}\sqrt{{\small x}_{n}{\small z}%
_{n}}}\right) \right] ,
\end{eqnarray*}%
where $x_{n}=W_{1}-rW_{n+1},\,y_{n}=W_{2}-rW_{n+2}-p(W_{1}-rW_{n+1}),\,%
\,z_{n}=-rtW_{n}$ , $j=-\dfrac{r(W_{2}-pW_{1})}{W_{1}}$ and 
\begin{eqnarray*}
f_{n} &=&\sum_{i=2}^{n}W_{i}e^{n-i}, \\
g_{n} &=&r\sum_{i=2}^{n-1}(W_{i+1}-pW_{i})e^{n-i}+W_{1}-prW_{n}, \\
h_{n}
&=&rt%
\sum_{i=1}^{n-3}W_{i}e^{n-1-i}+(W_{1}-r(pW_{n}+qW_{n-1})e+W_{2}-pW_{1}-qrW_{n}\,.
\end{eqnarray*}
\end{theorem}

\begin{proof}
Firstly, let us define $n$-square matrix 
\begin{equation}
F_{n}=\left( 
\begin{array}{cccccc}
1 & 0 & 0 & \ldots & 0 & 0 \\ 
0 & e^{n-2} & 0 & \ldots & 0 & 1 \\ 
0 & e^{n-3} & 0 & \ldots & 1 & 0 \\ 
0 & e^{n-4} & 0 &  &  & 0 \\ 
\vdots & \vdots & \vdots & {\mathinner{\mkern2mu\raise1pt\hbox{.}\mkern2mu
\raise4pt\hbox{.}\mkern2mu\raise7pt\hbox{.}\mkern1mu}} & \vdots & \vdots \\ 
0 & e & 1 & \ldots & 0 & 0 \\ 
0 & 1 & 0 & \ldots & 0 & 0%
\end{array}%
\right)  \label{25}
\end{equation}%
here $e$ is the positive root of the characteristic equation $%
x_{n}e^{2}+y_{n}e+z_{n}=0$, i.e., 
\begin{equation*}
e=\frac{-y_{n}+\sqrt{y_{n}^{2}-4x_{n}z_{n}}}{2x_{n}},
\end{equation*}%
where 
\begin{equation*}
x_{n}=W_{1}-rW_{n+1},\,y_{n}=W_{2}-rW_{n+2}-p(W_{1}-rW_{n+1}),\,\text{and}%
\,z_{n}=-rtW_{n}\,.
\end{equation*}%
Then, consider $n$-square matrix $G_{n}$ as below: 
\begin{equation*}
G_{n}=\left( 
\begin{array}{c|ccccccc}
1 & 0 & 0 & 0 & 0 & 0 & \cdots & 0 \\ 
-pr & 0 & 0 & 0 & 0 & 0 & \cdots & 0 \\ 
-qr & 0 & 0 & 0 & 0 & 0 & \cdots & 0 \\ 
-tr & 0 & 0 & 0 & 0 & 0 & \cdots & 0 \\ \hline
0 & 0 & 0 & 0 & 0 & 0 & \cdots & 1 \\ 
0 & 0 & 0 & 0 & 0 & 0 & \begin{picture}(2,2)
\multiput(0,0)(1.5,1){3}{.}\end{picture} & -p \\ 
0 & 0 & 0 & 0 & 0 & 1 & \begin{picture}(2,2)
\multiput(0,0)(1.5,1){3}{.}\end{picture} & -q \\ 
0 & 0 & 0 & 0 & 1 & -p & \begin{picture}(2,2)
\multiput(0,0)(1.5,1){3}{.}\end{picture} & -t \\ 
\vdots & \vdots & \begin{picture}(2,2)
\multiput(0,0)(1.5,1){3}{.}\end{picture} & \begin{picture}(2,2)
\multiput(0,0)(1.5,1){3}{.}\end{picture} & \begin{picture}(2,2)
\multiput(0,0)(1.5,1){3}{.}\end{picture} & \begin{picture}(2,2)
\multiput(0,0)(1.5,1){3}{.}\end{picture} & \begin{picture}(2,2)
\multiput(0,0)(1.5,1){3}{.}\end{picture} & \vdots \\ 
0 & 0 & 1 & -p & -q & -t & \cdots & 0 \\ 
0 & 1 & -p & -q & -t & 0 & \cdots & 0%
\end{array}%
\begin{array}{|ccc}
0 & 0 & 0 \\ 
0 & 0 & 1 \\ 
0 & 1 & -p \\ 
1 & -p & -q \\ \hline
-p & -q & -t \\ 
-q & -t & 0 \\ 
-t & 0 & 0 \\ 
0 & 0 & 0 \\ 
0 & 0 & 0 \\ 
\vdots & \vdots & \vdots \\ 
0 & 0 & 0%
\end{array}%
\right)
\end{equation*}%
It can be seen that for all $n>3,$%
\begin{equation*}
\det (G_{n})=\det (F_{n})=\left\{ 
\begin{array}{rl}
1, & n\equiv 1,2\text{ \ (}\func{mod}4\text{)} \\ 
-1, & n\equiv 1,2\text{ \ (}\func{mod}4\text{),}%
\end{array}%
\right.
\end{equation*}%
where $F_{n}$ is defined in \eqref{25} and $\det (G_{n}F_{n})=1$. By matrix
multiplication, we get; 
\begin{equation}
K_{n}=G_{n}E_{n}F_{n},  \label{50}
\end{equation}%
i.e., 
\begin{equation*}
{\small K}_{n}{\small =}\left( 
\begin{array}{cc|ccccc}
{\small W}_{1} & {\small f}_{n} & {\small W}_{n-1} & {\small W}_{n-2} & 
{\small \cdots } &  & {\small W}_{2} \\ 
r({\small W}_{2}{\small -pW}_{1}) & {\small g}_{n} & {\small %
r(W_{n}-pW_{n-1})} & {\small r(W_{n-1}-pW_{n-2})} & {\small \cdots } &  & r(%
{\small W}_{3}{\small -pW}_{2}) \\ 
{\small 0} & {\small h}_{n} & {\small W}_{1}{\small -r(pW}_{n}+{\small qW}%
_{n-1}) & {\small rtW}_{n-3} & {\small \cdots } &  & {\small rtW}_{1} \\ 
\hline
{\small 0} & {\small 0} & {\small y}_{n} & {\small x}_{n} &  &  &  \\ 
{\small 0} & {\small 0} & {\small z}_{n} & {\small y}_{n} & {\small x}_{n} & 
&  \\ 
&  &  & {\small \ddots } & {\small \ddots }\text{ \ } & {\small \ddots } & 
\\ 
{\small 0} & {\small 0} &  &  & \text{\ \ }{\small z}_{n}\text{ \ } & \text{%
\ \ }{\small y}_{n}\text{ \ } & \text{\ \ \ }{\small x}_{n}\text{ \ \ \ }%
\end{array}%
\right) ,
\end{equation*}%
where 
\begin{eqnarray*}
f_{n} &=&\sum_{i=2}^{n}W_{i}e^{n-i}, \\
g_{n} &=&W_{1}-prW_{n}+r\sum_{i=2}^{n-1}(W_{i+1}-pW_{i})e^{n-i}, \\
h_{n}
&=&W_{2}-pW_{1}-qrW_{n}\,+(W_{1}-r(pW_{n}+qW_{n-1})e+rt%
\sum_{i=1}^{n-3}W_{i}e^{n-1-i}.
\end{eqnarray*}%
\bigskip Multiplying the first row with $j=-\dfrac{r(W_{2}-pW_{1})}{W_{1}}$
and adding it to the second row in $K_{n}$, we obtain%
\begin{equation*}
\left\vert {\small K}_{n}\right\vert {\small =}\left\vert 
\begin{array}{ccccccc}
{\small W}_{1} & {\small f}_{n} & {\small W}_{n-1} & {\small W}_{n-2} & 
{\small \cdots } &  & {\small W}_{2} \\ 
{\small 0} & {\small g}_{n}{\small +jf}_{n} & r{\small W}_{n}{\small +(j-rp)W%
}_{n-1} & r{\small W}_{n-1}{\small +(j-rp)W}_{n-2} & {\small \cdots } &  & r%
{\small W}_{3}{\small +(j-rp)W}_{2} \\ 
{\small 0} & {\small h}_{n} & {\small W}_{1}{\small -r(pW}_{n}+{\small qW}%
_{n-1}) & {\small rtW}_{n-3} & {\small \cdots } &  & {\small rtW}_{1} \\ 
{\small 0} & {\small 0} & {\small y}_{n} & {\small x}_{n} & {\small \cdots }
&  & {\small 0} \\ 
{\small \vdots } & {\small \vdots } & {\small z}_{n} & {\small y}_{n} & 
{\small \ddots } &  & {\small \vdots } \\ 
&  &  & {\small \ddots } & {\small \ddots } & {\small \ddots } & {\small 0}
\\ 
{\small 0} & {\small 0} & {\small 0} & {\small 0} & {\small z}_{n}\text{ \ \ 
} & {\small y}_{n} & {\small x}_{n}%
\end{array}%
\right\vert .
\end{equation*}%
By Laplace expansion on the first column%
\begin{equation*}
\det {\small K}_{n}={\small W}_{1}\det Z_{n}=\det E_{n}
\end{equation*}%
here{\small 
\begin{equation*}
Z_{n}=\left( 
\begin{array}{c|cccccc}
{\tiny g}_{n}{\tiny +jf}_{n} & {\tiny rW}_{n}{\tiny +(j-rp)W}_{n-1} & {\tiny %
rW}_{n-1}{\tiny +(j-rp)W}_{n-2} &  & {\tiny \cdots } &  & {\tiny rW}_{3}%
{\tiny +(j-rp)W}_{2} \\ 
{\tiny h}_{n} & {\tiny W}_{1}{\tiny -r(pW}_{n}{\tiny +qW}_{n-1}{\tiny )} & 
{\tiny rtW}_{n-3} &  & {\tiny \cdots } &  & {\tiny rtW}_{1} \\ \hline
{\tiny 0} & {\tiny y}_{n} & {\tiny x}_{n} & {\tiny 0} &  &  & {\tiny 0} \\ 
{\tiny 0} & {\tiny z}_{n} & {\tiny y}_{n} & {\tiny x}_{n} &  &  &  \\ 
{\tiny \vdots } &  & {\tiny \ddots } & {\tiny \ddots } & \text{ }{\tiny %
\ddots }\text{ } & \text{ }{\tiny \ddots } & {\tiny 0} \\ 
{\tiny 0} &  &  & {\tiny 0} & \text{ \ \ }{\tiny z}_{n}\text{ \ \ } & \text{
\ \ }{\tiny y}_{n}\text{ \ \ } & {\tiny x}_{n}%
\end{array}%
\right) .
\end{equation*}%
}Applying Lemma 2, we complete the proof.
\end{proof}

\section{Inverse of $E_{n}$}

In this section, we compute the inverse of the matrix $E_{n}$. Note that,
just only for the inverse, we consider $W_{2}=pa.$ So,%
\begin{equation*}
G_{n}E_{n}F_{n}={\small K}_{n}{\small =}\left( 
\begin{array}{cc|ccccc}
{\small W}_{1} & {\small f}_{n} & {\small W}_{n-1} & {\small W}_{n-2} & 
{\small \cdots } &  & {\small W}_{2} \\ 
0 & {\small g}_{n} & {\small r(W_{n}-pW_{n-1})} & {\small r(W_{n-1}-pW_{n-2})%
} & {\small \cdots } &  & r({\small W}_{3}{\small -pW}_{2}) \\ 
{\small 0} & {\small h}_{n} & {\small W}_{1}{\small -r(pW}_{n}+{\small qW}%
_{n-1}) & {\small rtW}_{n-3} & {\small \cdots } &  & {\small rtW}_{1} \\ 
\hline
{\small 0} & {\small 0} & {\small y}_{n} & {\small x}_{n} &  &  &  \\ 
{\small 0} & {\small 0} & {\small z}_{n} & {\small y}_{n} & {\small x}_{n} & 
&  \\ 
&  &  & {\small \ddots } & {\small \ddots }\text{ \ } & {\small \ddots } & 
\\ 
{\small 0} & {\small 0} &  &  & \text{\ \ }{\small z}_{n}\text{ \ } & \text{%
\ \ }{\small y}_{n}\text{ \ } & \text{\ \ \ }{\small x}_{n}\text{ \ \ \ }%
\end{array}%
\right) .
\end{equation*}

\begin{lemma}
\cite{13} Let $\psi =\left( 
\begin{tabular}{ll}
$\alpha $ & $V$ \\ 
$U$ & $A$%
\end{tabular}%
\right) $ be an $(n-2)$-square matrix, then%
\begin{equation*}
\psi ^{-1}=\left( 
\begin{tabular}{cc}
$\frac{1}{l}$ & $-\frac{1}{l}VA^{-1}$ \\ 
$-\frac{1}{l}A^{-1}U$ & $A^{-1}+\frac{1}{l}A^{-1}UVA^{-1}$%
\end{tabular}%
\right) ,
\end{equation*}%
where $l=\alpha -VA^{-1}U,$ $V$ is a row vector and $U$ is a column vector.
\end{lemma}

\begin{lemma}
Let us define the matrix $T=\left[ t_{i,j}\right] _{i,j=1}^{n-3}$ of the
form:%
\begin{equation*}
t_{ij}=\left\{ 
\begin{array}{ll}
W_{1}x_{n} & ,~i=j, \\ 
W_{1}y_{n} & ,~i=j+1, \\ 
W_{1}z_{n} & ,~i=j+2, \\ 
0 & ,~otherwise.%
\end{array}%
\right.
\end{equation*}%
Then, inverse of $T$ is 
\begin{equation}
T^{-1}=\left[ t_{i,j}^{\prime }\right] _{i,j=1}^{n-3}=\left\{ 
\begin{tabular}{ll}
$\frac{1}{W_{1}x_{n}}$ & $,i=j$ \\ 
-$\frac{y_{n}}{W_{1}x_{n}^{2}}$ & $,i=j+1$ \\ 
$-\frac{y_{n}t_{i-2,j}^{\prime }+z_{n}t_{i-1,j}^{\prime }}{x_{n}}$ & $%
,i=j+k(k\geq 2)$ \\ 
$0$ & $,i<j.$%
\end{tabular}%
\right.  \label{7676}
\end{equation}
\end{lemma}

\begin{proof}
From matrix multiplication, we can easily see that $TT^{-1}=T^{-1}T=I_{n-3}$%
, where $I_{n-3}$ is identity matrix.
\end{proof}

\begin{theorem}
Let $E_{n}=circ_{r}\,(W_{1},W_{2},\ldots ,W_{n})$ be $r$-circulant matrix.
Then, 
\begin{equation*}
\begin{tabular}{l}
$E_{n}^{-1}=circ_{r}\,\left( c_{2}^{\prime }-\left( p+\frac{h_{n}}{g_{n}}%
\right) c_{3}^{\prime }-qc_{4}^{\prime }-tc_{5}^{\prime },\right. $ \\ 
$\ \ \ \ \ \ \ \ \ \ \ \ \ \ \ \ \ \ \ \ \ -pc_{2}^{\prime }+\left( \frac{%
ph_{n}}{g_{n}}-q\right) c_{3}^{\prime }-tc_{4}^{\prime },\frac{c_{n}^{\prime
}}{r},\frac{c_{n-1}^{\prime }-pc_{n}^{\prime }}{r}$ \\ 
$\ \ \ \ \ \ \ \ \ \ \ \ \ \ \ \ \ \ \ \ \left. \frac{1}{r}\left(
c_{n-2}^{\prime }-pc_{n-1}^{\prime }-qc_{n}^{\prime }\right) ,\ldots ,\frac{1%
}{r}\left( c_{n-k+3}^{\prime }-pc_{n-k+4}^{\prime }-qc_{n-k+5}^{\prime
}-tc_{n-k+6}^{\prime }\right) \right) ,$%
\end{tabular}%
\end{equation*}%
where%
\begin{equation*}
\begin{tabular}{l}
$c_{1}^{\prime }=0,$ \\ 
$c_{2}^{\prime }=W_{1}^{2}g_{n},$ \\ 
$c_{3}^{\prime }=-\frac{rW_{1}}{g_{n}}\overset{n-3}{\underset{k=0}{\sum }}%
s_{k}(W_{n-k}-pW_{n-k-1}),~(\text{for }s_{0}=\frac{1}{l}),$ \\ 
$c_{4}^{\prime }=-\frac{rp_{1}W_{1}(W_{n}-pW_{n-1})}{g_{n}}-\frac{rW_{1}}{%
g_{n}}\overset{n-3}{\underset{k=1}{\sum }}u_{k,1}(W_{n-k}-pW_{n-k-1}),$ \\ 
$\vdots $ \\ 
$c_{t}^{\prime }=-\frac{rp_{t-3}W_{1}(W_{n}-pW_{n-1})}{g_{n}}-\frac{rW_{1}}{%
g_{n}}\overset{n-3}{\underset{k=1}{\sum }}u_{k,t-3}(W_{n-k}-pW_{n-k-1}),~(t%
\geq 4)$%
\end{tabular}%
\end{equation*}%
and%
\begin{equation*}
\begin{tabular}{l}
$g_{n}=W_{1}-prW_{n}+r\sum_{i=2}^{n-1}(W_{i+1}-pW_{i})e^{n-i},$ \\ 
$h_{n}=W_{2}-pW_{1}-qrW_{n}\,+(W_{1}-r(pW_{n}+qW_{n-1})e+rt%
\sum_{i=1}^{n-3}W_{i}e^{n-1-i}.$%
\end{tabular}%
\end{equation*}
\end{theorem}

\begin{proof}
Firstly, Let us define 
\begin{equation*}
H_{n}=\left( 
\begin{tabular}{cccccc}
$1$ & $0$ & $0$ & $0$ & $\cdots $ & $0$ \\ 
$0$ & $1$ & $0$ & $0$ & $\cdots $ & $0$ \\ 
$0$ & $-\frac{h_{n}}{g_{n}}$ & $1$ & $0$ & $\cdots $ & $0$ \\ 
$0$ & $0$ & $0$ & $1$ & $\cdots $ & $0$ \\ 
$\vdots $ & $\vdots $ & $\vdots $ & $\vdots $ & ${\small \ddots }$ & $\vdots 
$ \\ 
$0$ & $0$ & $0$ & $0$ & $\cdots $ & $1$%
\end{tabular}%
\right)
\end{equation*}%
and%
\begin{equation*}
{\small L}_{n}{\small =}\left( 
\begin{tabular}{cccccc}
${\small W}_{1}$ & ${\small -f}_{n}$ & ${\small -W}_{n-1}{\small +}\frac{%
rf_{n}(W_{n}-pW_{n-1})}{g_{n}}$ & ${\small -W}_{n-2}{\small +}\frac{%
rf_{n}(W_{n-1}-pW_{n-2})}{g_{n}}$ & ${\small \cdots }$ & ${\small -W}_{n-2}%
{\small +}\frac{rf_{n}(W_{3}-pW_{2})}{g_{n}}$ \\ 
${\small 0}$ & ${\small W}_{1}$ & ${\small -}\frac{rW_{1}(W_{n}-pW_{n-1})}{%
g_{n}}$ & ${\small -}\frac{rW_{1}(W_{n-1}-pW_{n-2})}{g_{n}}$ & ${\small %
\cdots }$ & ${\small -}\frac{rW_{1}(W_{3}-pW_{2})}{g_{n}}$ \\ 
${\small 0}$ & ${\small 0}$ & ${\small W}_{1}$ & ${\small 0}$ & ${\small %
\cdots }$ & ${\small 0}$ \\ 
${\small 0}$ & ${\small 0}$ & ${\small 0}$ & ${\small W}_{1}$ & ${\small %
\cdots }$ & ${\small 0}$ \\ 
${\small \vdots }$ & ${\small \vdots }$ & ${\small \vdots }$ & ${\small %
\vdots }$ & ${\small \ddots }$ & ${\small \vdots }$ \\ 
${\small 0}$ & ${\small 0}$ & ${\small 0}$ & ${\small 0}$ & ${\small \cdots }
$ & ${\small W}_{1}$%
\end{tabular}%
\right) .
\end{equation*}%
Then, from matrix multiplication, we have%
\begin{equation*}
H_{n}G_{n}E_{n}F_{n}L_{n}=\left( 
\begin{tabular}{cccccc}
$W_{1}^{2}$ & $0$ &  &  &  &  \\ 
$0$ & $W_{1}g_{n}$ &  &  &  &  \\ 
&  & $W_{1}\rho _{3}$ & $W_{1}\rho _{4}$ & $\cdots $ & $W_{1}\rho _{n}$ \\ 
&  & $W_{1}y_{n}$ & $W_{1}x_{n}$ &  &  \\ 
&  & $W_{1}z_{n}$ & $W_{1}y_{n}$ & ${\small \ddots }$ &  \\ 
&  &  & ${\small \ddots }$ & ${\small \ddots }$ & $W_{1}x_{n}$ \\ 
&  &  &  & $W_{1}z_{n}$ & $W_{1}y_{n}$%
\end{tabular}%
\right) =\mathcal{Y}_{1}\oplus N,
\end{equation*}%
where $\mathcal{Y}_{1}=diag(W_{1}^{2},W_{1}g_{n}),$ $\mathcal{Y}_{1}\oplus N$
is the direct sum of $\mathcal{Y}_{1}$ and $N,$%
\begin{equation*}
\rho _{3}=W_{1}-r\left( W_{n}\left( p+\frac{h_{n}}{g_{n}}\right)
-W_{n-1}\left( q+p\frac{h_{n}}{g_{n}}\right) \right)
\end{equation*}%
and 
\begin{equation*}
\rho _{i}=-\frac{rh_{n}}{g_{n}}(W_{n-i+3}-pW_{n-i+2})+rtW_{n-i+1}\text{ \ \
\ \ \ for }i=4,5,\ldots ,n\text{.}
\end{equation*}%
If we define $P=H_{n}G_{n}$ and $Q=F_{n}L_{n}$, we get%
\begin{equation*}
E_{n}^{-1}=Q\left( \mathcal{Y}_{1}^{-1}\oplus N^{-1}\right) P.
\end{equation*}%
According to Lemma 4, we define $(n-2)$-square matrix%
\begin{equation*}
N=\left( 
\begin{tabular}{ll}
$W_{1}\rho _{3}$ & $V$ \\ 
$U$ & $T$%
\end{tabular}%
\right) .
\end{equation*}%
Then, we have 
\begin{equation*}
N^{-1}=\left( 
\begin{tabular}{cc}
$\frac{1}{l}$ & $\frac{-VT^{-1}}{l}$ \\ 
$\frac{-T^{-1}U}{l}$ & $T^{-1}+\frac{1}{l}T^{-1}UVT^{-1}$%
\end{tabular}%
\right) ,
\end{equation*}%
where 
\begin{eqnarray*}
U &=&\left( W_{1}y_{n},W_{1}z_{n},0,\ldots ,0\right) ^{T}, \\
V &=&(W_{1}\rho _{4},W_{1}\rho _{5},\ldots ,W_{1}\rho _{n}), \\
T &=&\left\{ 
\begin{tabular}{ll}
$W_{1}x_{n}$ & $,i=j$ \\ 
$W_{1}y_{n}$ & $,i=j+1$ \\ 
$W_{1}z_{n}$ & $,i=j+2$ \\ 
$0$ & $,otherwise,$%
\end{tabular}%
\right. \\
l &=&W_{1}\left( \rho _{3}-W_{1}y_{n}\overset{n-3}{\underset{i=1}{\dsum }}%
t_{i1}\rho _{i+3}-W_{1}z_{n}\overset{n-4}{\underset{i=1}{\dsum }}\rho
_{i+4}\right) .
\end{eqnarray*}%
Let be $R=\frac{-VT^{-1}}{l}$ row vector$,~S=\frac{-T^{-1}U}{l}$ column
vector and $J=T^{-1}+\frac{1}{l}T^{-1}UVT^{-1},$ where $T^{-1}$ is as in %
\eqref{7676} Then, we have%
\begin{equation*}
R=[p_{1},p_{2},\ldots ,p_{n-3}],
\end{equation*}%
where $p_{i}=\frac{-W_{1}}{l}\underset{k=i}{\overset{n-3}{\sum }}\rho
_{k+3}t_{k,i}^{\prime }~,$%
\begin{equation*}
S=[s_{1},s_{2},\ldots ,s_{n-3}]^{T},
\end{equation*}%
where $s_{1}=\frac{-W_{1}}{l}t_{1,1}^{\prime }y_{n}~$and for $i\geq 2,$ $%
s_{i}=\frac{-W_{1}}{l}\left( y_{n}t_{i,1}^{\prime }+z_{n}t_{i,2}^{\prime
}\right) ,$%
\begin{equation*}
J=u_{i,j}=\left\{ 
\begin{tabular}{ll}
$t_{1,j}^{\prime }-\frac{W_{1}^{2}}{l}y_{n}t_{1,1}^{\prime }\underset{k=j}{%
\overset{n-3}{\sum }}\rho _{k+3}t_{k,j}^{\prime }~$ & , for $i=1$ \\ 
$t_{i,j}^{\prime }-\frac{W_{1}^{2}}{l}(y_{n}t_{i,1}^{\prime
}+z_{n}t_{i,2}^{\prime })\underset{k=j}{\overset{n-3}{\sum }}\rho
_{k+3}t_{k,j}^{\prime }$ & , for $i=2,3,\ldots ,n-3$.%
\end{tabular}%
\right.
\end{equation*}%
So, we obtain%
\begin{equation*}
N^{-1}=\left( 
\begin{tabular}{lllll}
$\frac{1}{l}$ & $p_{1}$ & $p_{2}$ & $\cdots $ & $p_{n-3}$ \\ 
$s_{1}$ & $u_{1,1}$ & $u_{1,2}$ & $\cdots $ & $u_{1,n-3}$ \\ 
$s_{2}$ & $u_{2,1}$ & $u_{2,2}$ & $\cdots $ & $u_{2,n-3}$ \\ 
$\vdots $ & $\vdots $ & $\vdots $ & $\ddots $ & $\vdots $ \\ 
$s_{n-3}$ & $u_{n-3,1}$ & $u_{n-3,2}$ & $\cdots $ & $u_{n-3,n-3}$%
\end{tabular}%
\right) _{(n-2)\times (n-2)},
\end{equation*}%
where $s_{i}$'s, $p_{i}$'s and $u_{i,j}$'s are as in above.

The last row elements of the $Q=F_{n}L_{n}$ are $0,~W_{1},$ $-\frac{%
rW_{1}(W_{n}-pW_{n-1})}{g_{n}},~-\frac{rW_{1}(W_{n-1}-pW_{n-2})}{g_{n}}%
,\ldots ,~-\frac{rW_{1}(W_{3}-pW_{2})}{g_{n}}$. Then, the last row elements
of $Q\left( \mathcal{Y}_{1}^{-1}\oplus N^{-1}\right) $ are as the following:%
\begin{eqnarray*}
c_{1}^{\prime } &=&0, \\
c_{2}^{\prime } &=&W_{1}^{2}g_{n}, \\
c_{3}^{\prime } &=&-\frac{rW_{1}}{g_{n}}\overset{n-3}{\underset{k=0}{\sum }}%
s_{k}(W_{n-k}-pW_{n-k-1}),~(\text{for }s_{0}=\frac{1}{l}), \\
c_{4}^{\prime } &=&-\frac{rp_{1}W_{1}(W_{n}-pW_{n-1})}{g_{n}}-\frac{rW_{1}}{%
g_{n}}\overset{n-3}{\underset{k=1}{\sum }}u_{k,1}(W_{n-k}-pW_{n-k-1}), \\
&&\vdots \\
c_{t}^{\prime } &=&-\frac{rp_{t-3}W_{1}(W_{n}-pW_{n-1})}{g_{n}}-\frac{rW_{1}%
}{g_{n}}\overset{n-3}{\underset{k=1}{\sum }}%
u_{k,t-3}(W_{n-k}-pW_{n-k-1}),~(t\geq 4).
\end{eqnarray*}%
Since inverse of $r$-circulant matrix is $r$-circulant matrix \cite{4}, $%
E_{n}^{-1}$ matrix is an $r$-circulant matrix. If $E_{n}^{-1}=circ_{r}\left(
c_{1},c_{2},\ldots ,c_{n}\right) $, last row elements of the $E_{n}^{-1}$
matrix are as in below:%
\begin{equation*}
\begin{tabular}{l}
$rc_{2}=-prc_{2}^{\prime }+\left( \frac{prh_{n}}{g_{n}}-qr\right)
c_{3}^{\prime }-trc_{4}^{\prime }$ \\ 
$rc_{3}=c_{n}^{\prime }$ \\ 
$rc_{4}=c_{n-1}^{\prime }-pc_{n}^{\prime }$ \\ 
$rc_{5}=c_{n-2}^{\prime }-pc_{n-1}^{\prime }-qc_{n}^{\prime }$ \\ 
$\vdots $ \\ 
$rc_{k}=c_{n-k+3}^{\prime }-pc_{n-k+4}^{\prime }-qc_{n-k+5}^{\prime
}-tc_{n-k+6}^{\prime }$ $\ \ \ \ \left( \text{for}~5<k\leq n\right) ,$ \\ 
$c_{1}=c_{2}^{\prime }-\left( p+\frac{h_{n}}{g_{n}}\right) c_{3}^{\prime
}-qc_{4}^{\prime }-tc_{5}^{\prime }.$%
\end{tabular}%
\end{equation*}%
Therefore, we complete this proof.
\end{proof}

\end{document}